\documentclass[a4paper,12pt]{amsart}
\usepackage{amssymb}
\usepackage{bbm}
\usepackage[T1]{fontenc}
\usepackage[utf8]{inputenc}
\usepackage[dvips]{graphicx}
\usepackage{pdfsync}
\usepackage{color,xcolor}
\usepackage{verbatim}
\usepackage{enumitem}
\usepackage[margin=2cm]{geometry}
\newtheorem{thm}{Theorem}[section]
\newtheorem*{thm*}{Theorem}
\newtheorem{lem}[thm]{Lemma}
\newtheorem*{lem*}{Lemma}
\newtheorem{fct}[thm]{Fact}
\newtheorem*{fct*}{Fact}
\newtheorem{cor}[thm]{Corollary}
\newtheorem*{cor*}{Corollary}

\newtheorem*{prop*}{Proposition}
\theoremstyle{definition}
\newtheorem{defn}[thm]{Definition}
\newtheorem*{defn*}{Definition}
\theoremstyle{remark}

\newtheorem*{rem*}{Remark}

\newtheorem*{example*}{Example}

\newtheorem*{que*}{Question}

\newcommand{\N}{\mathbb N}

\newcommand{\abs}[1]{\left\vert#1\right\vert}
\newcommand{\set}[1]{\left\{#1\right\}}
\newcommand{\eps}{\varepsilon}

\newcommand{\CA}{\mathcal{A}}
\newcommand{\CB}{\mathcal{B}}
\newcommand{\CC}{\mathcal{C}}

\newcommand{\CM}{\mathcal{M}}

\newcommand{\CT}{\mathcal{T}}

\newcommand{\CP}{\mathcal{P}}
\newcommand{\CQ}{\mathcal{Q}}

\newcommand{\CS}{\mathcal{S}}

\renewcommand{\emptyset}{\varnothing}

\newcommand{\sd}{\bigtriangleup}

\newcommand{\Folner}{F\o lner }

\usepackage{mathrsfs}

\renewcommand{\epsilon}{\varepsilon}

\renewcommand{\leq}{\leqslant}
\renewcommand{\geq}{\geqslant}

\title[Residuality of dynamical morphisms]{Residuality of dynamical morphisms for amenable group actions}

\author{Dawid Huczek, Sebastian Kopacz, Jacek Serafin}

\address{\hskip- \parindent
Dawid Huczek, \emph{ Faculty of Pure and Applied Mathematics, Wroclaw University of Science and Technology, Wybrze\.ze Wyspia\'nskiego 27, 50-370 Wroc\l aw, Poland}}
\email{dawid.huczek@pwr.edu.pl}

\address{\hskip- \parindent
	Sebastian Kopacz, \emph{Faculty of Pure and Applied Mathematics, Wroclaw University of Science and Technology, Wybrze\.ze Wyspia\'nskiego 27, 50-370 Wroc\l aw, Poland}}
\email{sebastian.kopacz@pwr.edu.pl}

\address{\hskip- \parindent
	Jacek Serafin, \emph{Faculty of Pure and Applied Mathematics, Wroclaw University of Science and Technology, Wybrze\.ze Wyspia\'nskiego 27, 50-370 Wroc\l aw, Poland}}
\email{jacek.serafin@pwr.edu.pl}



\date{\today}

\begin{document}

\begin{abstract}
We extend the classical Baire category approach, used in proving the finite generator theorem of Krieger, the homomorphism theorem of Sinai and the isomorphism theorem of Ornstein, applying a similar reasoning to the case of actions of countably infinite amenable groups.
In principle we follow the lines of the paper by Burton, Keane and Serafin (\cite{BKS}), showing that measures defining homomorphisms or isomorphisms form residual subsets in suitably chosen spaces of joinings.

\end{abstract}
\maketitle

\section{Introduction}
The finite generator theorem of Krieger \cite{Kr}, the homomorphism theorem of Sinai \cite{Si} and the Ornstein isomorphism theorem for Bernoulli shifts \cite{O}, are  undoubtedly among the most significant results in ergodic theory. Classically, all of these results established the existence of an isomorphism or homomorphism between dynamical systems governed by the action of a single (usually invertible) measure-preserving transformation (a $\mathbb{Z}$-action), and the proofs relied on the entropy techniques of $\mathbb{Z}$-actions. 
Entropy theory has later been (and is still being) extended to actions of more general groups, with countably infinite amenable group appearing to be the most natural setting. In such generality the Kolmogorov-Sinai entropy was extended by Kieffer \cite{Ki}, the isomorphism theory of Bernoulli shifts (as well as such fundamental notions as the $\overline{d}$ metric and finite determinedness) was extended by Ornstein and Weiss \cite{OW}, and the pointwise ergodic theorem was proved by Lindenstrauss \cite{L}, just to mention some of the most important results. Since then, there have been further developments (beyond actions of amenable groups) in entropy theory, let us only note here the sofic group action results contained in the groundbreaking works by Lewis Bowen \cite{B1}, \cite{B2} and Kerr and Li \cite{KL1}.
As our article addresses actions of countable amenable groups, we will not elaborate on the details of more general group actions, instead referring the reader to a recent book of Kerr and Li \cite{KL}.
The fact that the three milestones invoked at the beginning of this section are not listed chronologically, is not accidental. In 1977 (in an unpublished note of \cite{BR}), Burton and Rothstein established a novel approach connecting the three theorems; they realized it was easier to argue that there were ``many'' isomorphisms (or homomorphisms) in a suitable Baire space of measures (where certain measures can be identified with isomorphisms) than to construct a particular isomorphism or homomorphism directly. Their method is first used to establish the finite generator theorem which plays the central role. Then the generator theorem and a clever application of the $\overline{d}$ metric and finite determinedness, are used to prove the Sinai theorem, and finally, the Ornstein theorem becomes a trivial consequence of the Sinai theorem. All of this is possible as the intersection of two residual sets is still residual, in a Baire space. Later, Rudolph presented the Burton and Rothstein method in his book \cite{R} (see also \cite{Se} for more details). In 2000 Burton, Keane and Serafin \cite{BKS}, published a more elementary version of \cite{BR}, whereas the classical treatise of Glasner \cite{G} contains all details of Rudolph’s exposition. Needless to say, all of the above works address the $\mathbb{Z}$-actions. It is also worth mentioning that a slightly different approach is presented in \cite{DS}, where the $\overline{d}$ metric and finite determinedness are replaced by more elementary tools.
In the present paper we follow similar lines to those of \cite{BKS}, although we aim to provide more detail, depth and clarity. We prove the residual  theorems of Krieger, Sinai and Ornstein, in the case of the action of a countable amenable group. Again, the finite generator theorem plays a central role, but its proof requires overcoming difficulties absent in the case of $\mathbb{Z}$-actions. In particular the notion of markers becomes much more intricate: we need to use tools such as quasitilings and dynamical tilings, originally introduced by Ornstein and Weiss, later extended by Downarowicz, Huczek and Zhang \cite{DHZ}, and then by Conley et al. \cite{CJKMST}. In addition, we need to develop marker blocks allowing us to encode a quasitiling structure using some coordinates of a symbolic element, leaving enough freedom at other coordinates to achieve other necessary properties. Other than that, we hope that our considerations remain elementary and relatively easy to follow.

\section{Preliminaries}
\subsection{Amenable groups and their actions}

\subsubsection{F\o lner sets and invariance properties}
Throughout this paper, by $G$ we will denote a countably infinite amenable group, and $(F_n)$ will denote a fixed F\o lner sequence, i.e. a sequence of finite subsets of $G$ such that for every $g\in G$ the sequence $\frac{\abs {gF_n\sd F_n}}{\abs{F_n}}$ tends to $0$ as $n$ goes to infinity. Multiplication involving sets will always be understood element-wise, so $gF_n$ is the set $\set{gf: f\in F_n}$, and $KF=\set{kf: k\in K, f\in F}$. The conclusions of our theorems do not inherently rely on the choice of the F\o lner sequence, thus we can assume without loss of generality that the F\o lner sequence we fix for $G$ is \emph{tempered}, i.e. there exists some $C>0$ such that for all $n$ we have $\abs{\bigcup_{k<n}F_kF_n}\leq C\abs{F_n}$ (we will not use this condition directly, but it is needed for the pointwise ergodic theorem and the Shannon-McMillan-Breiman theorem to hold). For convenience we can also assume that our F\o lner sets are increasing, symmetric, contain the identity element $e\in G$, and their union is all of $G$. It is well known that every countable amenable group has a F\o lner sequence satisfying all the above properties.


\subsubsection{Actions, symbolic dynamical systems and quasitilings}
By an amenable group action we understand a pair $(X,G)$, where $X$ is a compact metric space, and $G$ acts on $X$ by homeomorphisms. If $x$ is an element of $X$ or $A$ is a subset of $X$, we will sometimes write $gx$ and $gA$ instead of $g(x)$ and $g(A)$, respectively. We say the action of $G$ on $X$ is \emph{free} if for every $x\in X$ and every $g\in G$, other than the identity element, we have $gx\neq x$; the action is \emph{ergodic} if for every $G$-invariant set $A$ we have $\mu(A)=0$ or $\mu(A)=1$. Our main concern will be group actions on \emph{symbolic spaces}: for a finite or countable set $\Lambda$, $X$ will be a subset of the space $\Lambda^G$ (with product topology), closed and invariant under the right-shift action defined by $gx(h)=x(hg)$.  For any finite $D\subset G$, a \emph{block with domain $D$} is a mapping from $D$ into $\Lambda$. We say that a block $B$ with domain $D$ occurs in $x\in X$ (at position $g$) if for every $d\in D$ we have $x(dg)=B(d)$. We will denote this concisely by $x(Dg)=B$. In addition, a \emph{cylinder} determined by $B$ is the set of all $x\in X$ such that $x(D)=B$. This correspondence between blocks and subsets of $X$ allows us to speak of measures of blocks (meaning the measures of the corresponding cylinders).
 
A \emph{quasitiling} is a family $\CT=\set{Sc: S\in\CS(\CT),c\in C(S)}$ of finite subsets of $G$ (referred to as \emph{tiles}), where $\CS(\CT)$ is a finite collection of finite subsets of $G$ (referred to as \emph{shapes}), and for each $S\in\CS(\CT)$, $C(S)$ is a subset of $G$ (not necessarily finite), whose elements are referred to as \emph{centers}. We assume that for every $T\in \CT$, the representation $T=Sc$ is unique, and that $C(S_1)$ and $C(S_2)$ are disjoint whenever $S_1\neq S_2$. Moreover we assume that each shape contains the group identity element $e$. Every quasitiling $\CT$ of $G$ with finitely many shapes can be understood as an element of a symbolic dynamical system (with the action of $G$), where the alphabet is $\Lambda = \set{\emptyset}\cup\CS(\CT)$ and $\CT$ corresponds to an element $x_\CT \in \Lambda^G$ defined by $x_\CT(g)=S$ if $Sg\in\CT$, and $x_\CT(g)=\emptyset$ otherwise. This allows us to discuss dynamical notions (such as factorization or entropy) in the context of quasitilings, by applying these notions to the orbit closures of the corresponding symbolic elements.

If $T$ is a finite subset of $G$, $T'\subset T$, and $\alpha\in(0,1)$, we say that $T'$ is an $\alpha$-subset of $T$, if $\abs{T'}>\alpha\abs{T}$.
Recall that the \emph{lower Banach density} of a set $A\subset G$ is defined as
\[
\underline{D}(A) = \lim_{n\to\infty} \inf_{g\in G} \frac{|A\cap F_ng|}{|F_n|}. 
\]
Let $\epsilon \in [0,1)$ and $\alpha \in (0,1]$. A quasitiling $\CT$ is called:
\begin{itemize}
    \item $\epsilon-disjoint$ if there exists a mapping $T \mapsto T^{\circ}$ where $T\in\CT$ such that $T^{\circ} \subseteq T$, $T^{\circ}$ is a $(1-\epsilon)$-subset of $T$ and if $T' \neq T$, then $T^{'\circ} \cap T^{\circ} = \emptyset$;
    \item \textit{disjoint} if the tiles of $\CT$ are pairwise disjoint;
    \item $\alpha-covering$ if $\underline{D}(\bigcup\CT) \geq \alpha$.
\end{itemize}

Our quasitilings will be obtained mainly from the following lemma, which is essentially a restatement of the results of \cite{CJKMST} in terms of quasitilings rather than towers:
\begin{lem}
\label{lem:dyntil}
    Let $X$ be a compact metric space with a free action of a countable amenable group $G$. For any $\eta>0$ and any $K\in\N$ there exists a collection of quasitilings $\set{\CT_x : x\in X}$ such that:
    \begin{enumerate}
        \item For every $x\in X$, $\CT_x$ is a disjoint, $(1-\eta)$-covering quasitiling of $G$.
        \item There exist F\o lner sets $F_{k_1},F_{k_2},\ldots, F_{k_n}$, such that for every $x\in X$ and every $T\in \CT_x$, $T$ is a $(1-\eta)$-subset of $F_{k_i}$ for some $i$.
        \item In the preceding property, we have $K\leq k_1\leq k_2\leq\ldots k_n$, and $n$ depends only on $\eta$ (specifically, $n$ is the smallest integer such that $(1-\eta)^n<\eta$).
        \item The mapping $x\mapsto \CT_x$ is a Borel function (from $X$ into the set of all quasitilings whose shapes are subsets of $\bigcup_{i=1}^n F_{k_i}$), such that $\CT_{gx}=\set{T{g^{-1}}: T\in \CT_x}$. In the symbolic interpretation, where we consider $\CT_x$ as an element of the full shift $\set{0,\ldots,n}^G$ (defined as $\CT_x(h)=i \iff F_{k_i}h\in \CT_x$, and $\CT_x(h)=0$ if no such $i$ exists), this means that the mapping $x\mapsto \CT_x$ is a factor map from $X$ into $\set{0,\ldots,n}^G$.
    \end{enumerate}
\end{lem}
\begin{proof}

The first part of the construction of $\CT_x$ exactly mimics the proof of lemma 3.4 of \cite{CJKMST}: note that the sets $F_1,\ldots,F_n$ in \cite{CJKMST} are only required to be sufficiently invariant with respect to each other and some fixed subset of the group, and thus they can be assumed to be a subsequence of the F\o lner sequence beginning at a specified index (or later). Thus we obtain the $n$ F\o lner sets $F_{k_1},\ldots F_{k_n}$, and corresponding Borel sets $C_1,C_2,\ldots C_n\subset X$ (this is the major reason why we do not apply the lemma 3.4 directly in the present paper: the statement of lemma 3.4 already refers to the union of these sets, and we need to refer to them individually), such that $\bigcup_{i=1}^n\bigcup_{c\in C_i}F_ic$ has lower Banach density (as defined for subsets of $X$ in \cite{CJKMST}) at least $1-\eta$. 

Now, for every $x \in X$ we can define $\CT'_x=\bigcup_{i=1}^n\set{F_{k_i}h: hx\in C_i}$. Intuitively, if we interpret $C_1,C_2,\ldots,C_n$ as ``marker sets'', then the centers of the tiles of $\CT'_x$ correspond to times of visits of $x$ in these marker sets.

This gives us a collection of quasitilings which we will show to have all the desired properties, with the exception of disjointness (the actual system  of quasitilings $\CT_x$, satisfying all the required properties of the lemma, will be obtained as a factor of this quasitiling system, so all the properties we establish for $\CT'_x$ will carry over). Properties $(2)$ and $(3)$ are ensured by the choice of $F_{k_1},\ldots F_{k_n}$. Property $(4)$ follows from the following calculation: if $T\in \CT'_x$ then $T = F_{k_i}h$ for some $h$ such that $hx \in C_i$, and thus $hg^{-1}gx \in C_i$. Consequently, we have $F_{k_i}hg^{-1}\in \CT'_{gx}$, which is equivalently stated $Tg^{-1} \in \CT'_{gx}$.

Now, we must show that if we interpret $\CT'_x$ as an element of the full shift $\set{0,1,\ldots,n}^G$ (as defined in the statement of the lemma --- note that since the sets $C_1,\ldots,C_n$ are pairwise disjoint, the symbolic representation is unambiguous), then $g\CT'_x$ corresponds to $\CT'_{gx}$. Note the following: $\CT'_x(h) = i \iff F_{k_i}h \in \CT'x \iff F_{k_i}hg^{-1}\in \CT'_{gx}$. On the other hand $\CT'_x(h)=i \iff g\CT'_x(hg^{-1})=i$, which means that $g\CT'_x$ corresponds to a quasitiling which has $F_{k_i}hg^{-1}$ as a tile. This argument holds for every $g$, $h$, and $i$, and thus the quasitiling corresponding to $g\CT'_x$ is exactly $\CT'_{gx}$. In addition, since for every $i$ we have $\set{x\in X: \CT'_x(e)=i} = C_i$ (where $e$ denotes the neutral element of $G$), the mapping $x\mapsto\CT'_x$ is Borel.




To establish that for every $x$ the quasitiling $\CT'_x$ has lower Banach density at least $1-\eta$ we need some clarification, because \cite{CJKMST} uses a definition of LBD for subsets of the dynamical system, and the present paper uses a definition of LBD for subsets of the group itself. However, note that if we denote the union of the tiles of $\CT'_x$ as $T'_x$, then the condition $h\in T'_x$ is equivalent to saying that $h\in F_{k_i}g$ for some $i$ and some $g$ such that $gx\in C_i$. In other words $hx\in F_{k_i}gx = F_{k_i}c$ for some $c\in C_i$, so $hx \in \bigcup_c{F_cc}$ (using the notation of lemma 3.4 in \cite{CJKMST}). Thus we have $h\in T'_x$ if and only if $hx\in \bigcup_c{F_cc}$. In addition, for any finite $F\subset G$ and any $g\in G$ we have $h\in Fg$ if and only if $hx\in Fgx$, so ultimately, $\abs{T'_x\cap F} = \abs{\bigcup_c{F_cc} \cap Fgx}\geq \inf\limits_{ y\in X}{\abs{\bigcup_c{F_cc} \cap Fy}}$, which establishes that the lower Banach density of $\CT'_x$ is bounded from below by the lower Banach density of $\bigcup_c{F_cc}$ (and thus it is at least $1-\eta$).

All that remains is to modify $\CT'_x$ to ensure disjointness of the tiles. Since for every $c\in C_i$ there exists a $(1-\eta)$-subset $F_c$ of $F_{k_i}$ and the sets $F_cc$ are pairwise disjoint (over all $i$ and all $c\in C_i$), this means that we can replace all tiles of every $\CT'_x$ by their relatively large subsets (of relative size at least $1-\eta$), obtaining a disjoint quasitiling $\CT_x$. Specifically, the mapping $\CT'_x\mapsto \CT_x$ can be defined as follows:
\begin{enumerate}
    \item Enumerate all the elements of $F= \bigcup_i F_{k_i}$ as $f_1,f_2,\ldots f_J$.
    \item If $h\in G$ belongs to multiple tiles of $\CT_x$, then for each such tile $T$ it can be written as $f_{j_T}g_T$ for some $g_T\in G$ (specifically, $g_T$ is such that $g_Tx\in C_i$ for some $i$) and $j_T\leq J$. Furthermore, if $T\neq T'$, then $j_T\neq j_{T'}$, since otherwise the tiles  of $\CT_x$ would not be $\eta$-disjoint.
    \item Based on the above, assign such an $h$ to the tile $T$ for which $j_T$ is the smallest number, and remove it from the other tiles.
\end{enumerate}
Since the procedure described above commutes with the action of $G$, and for every $h$ it depends only on coordinates from $F^{-1}Fh$, it can be understood as a factor map (block code) from $\CT'_x$ to $\CT_x$, yielding the disjoint quasitiling we needed.
\end{proof}

\subsection{Partitions, blocks, and measures}\phantom{ }\\
Throughout this subsection, we will consider a symbolic dynamical system $(X,\CB,\mu, G)$ with finite alphabet $\Lambda$ and an ergodic free action of a countable amenable group $G$ by the right shift.

\begin{defn}
If $P$ and $Q$ are measurable (finite or countable) partitions of $X$, and $\eps>0$, then we write
$P \subset_{\epsilon} Q \text{ mod } \mu$, if for every $A\in P$ there exists a $B$ which is the union of elements of $Q$, and such that $\mu(A\sd B)<\eps$. Analogously, we define the inclusion $P \subset  Q \text{ mod } \mu,$
replacing the condition $\mu(A\sd B)<\eps$ with $\mu(A\sd B)=0$,
and we write $P = Q \text{ mod } \mu,$
if $P \subset  Q \text{ mod } \mu$ and $Q \subset  P \text{ mod } \mu$.
\end{defn}

The weak-$*$ topology on the set of all probability measures on $(X,\CB)$ is metrizable in a number of equivalent ways. We will explicitly use the following metric:
\[ d(\mu_1,\mu_2)=\sum_{n=1}^\infty\frac{1}{2^n}\frac{1}{\abs{\CB_n}}\sum_{B\in\CB_n}\abs{\mu_1(B)-\mu_2(B)}, \]
where $\CB_n$ is the set of all blocks over $\Lambda$ with domain $F_n$ (so $\abs{\CB_n}=\abs{\Lambda}^{\abs{F_n}}$). We will also often use the following argument for estimating proximity between two measures:
\begin{fct}
    For any $\eps>0$ there exist $\delta>0$ and $n>0$, such that if $\abs{\mu_1(B)-\mu_2(B)}<\delta$ for all $B\in \CB_n$, then $d(\mu_1,\mu_2)<\eps$.
\end{fct}

Given a pair of blocks $B'$ and $B$ (with domains $D'$ and $D$ respectively), the frequency of $B'$ in $B$ is defined as follows:
\[
    \text{fr}_{B}(B') = \frac{1}{|D|}|\{h\in D:D'h\subset D,\ B(gh)=B'(g) \text{ for all }g\in D'\}|.
\]
The metric defined above on the set of measures can be extended to blocks as follows:
\[ d(B_1,B_2)=\sum_{n=1}^\infty\frac{1}{2^n}\frac{1}{\abs{\CB_n}}\sum_{B\in\CB_n}\abs{\text{fr}_{B_1}(B)-\text{fr}_{B_2}(B)}, \]
\[ d(\mu,C)=\sum_{n=1}^\infty\frac{1}{2^n}\frac{1}{\abs{\CB_n}}\sum_{B\in\CB_n}\abs{\mu(B)-\text{fr}_{C}(B)}. \]

Recall the pointwise ergodic theorem for countable (discrete) groups (see e.g. \cite{L}). 
\begin{thm}
\label{lem:PET for discreate group}
    Let $G$ be a countable amenable group acting on a measure space $(X, \mathcal{B},\mu)$ by measure preserving transformations, and let $(F_n)$ be a tempered \Folner sequence. Then for any $f\in L^1(\mu)$, there exists a $G$-invariant $\overline{f}\in L^1(\mu)$ such that
    \[
        \lim_{n\to\infty} \frac{1}{|F_n|}\sum_{g\in F_n} f(gx) = \overline{f}(x) \text{ a.e } .
    \]
    In particular, if the $G$ action is ergodic,
    \[
        \lim_{n\to\infty} \frac{1}{|F_n|}\sum_{g\in F_n} f(gx) = \int f(x) d\mu(x) \text{ a.e } .
    \]
\end{thm}
As a consequence, if $B$ is a block in $X$ and we set $\delta_B$ to be the indicator function of the corresponding cylinder, we get the following:
\[
    \lim_{n\to\infty} \frac{1}{|F_n|}\sum_{g\in F_n} \delta_B(gx) = \int \delta_B(x) d\mu(x) = \mu(B) \text{ a.e } .
\]
Equivalently, for $\mu$-almost every $X$, we have $\lim_{n\to\infty}\text{fr}_{x(F_n)}(B)=\mu(B)$, which can be also stated as $\lim_{n\to\infty}d( \mu, x(F_n))=0$, and thus we can approximate the measure of any cylinder by the frequency with which the corresponding block occurs in $x(F_n)$ for large enough $n$. In addition, since pointwise convergence in the ergodic theorem implies convergence in measure, we have the following:
\begin{cor}
If $\mu$ is ergodic, then for every $\delta_1,\delta_2>0$ there exists an $N$ such that for $n>N$ the collection of blocks $B$ with domain $F_n$ such that $d(\mu,B)<\delta_1$ has joint measure greater than $1-\delta_2$.
\end{cor}

A straightforward estimate will allow us to replace $F_n$ in the above corollary by a sufficiently large subset:
\begin{lem}
\label{lem: PET for subsets}
    Suppose for some $\delta>0$ and finite $F\subset G$ we have
    \[
         \mu(B)- \delta \leq \frac{1}{|F|}\sum_{g\in F} \delta_B(gx) \leq \mu(B) + \delta.
    \]
    If $\eta < \frac{\delta}{1+2\delta}$ and $F'$ is a $(1-\eta)$-subset of $F$ then
    \[
        \mu(B) - 2\delta \leq \frac{1}{|F'|}\sum_{g\in F'} \delta_B(gx) \leq \mu(B) + 2\delta.
    \]
\end{lem}
\begin{proof}
    From one side, we have
    \[\begin{split}
        \frac{1}{|F'|}\sum_{g\in F'}\delta_B(gx) &\geq \frac{1}{|F|} \Big(\sum_{g\in F} \delta_B(gx) - \sum_{g\in F\setminus F'} \delta_B(gx) \Big)\geq \\ &\geq \frac{1}{|F|} \Big(\sum_{g\in F} \delta_B(gx) - |F\setminus F'| \Big) \geq \mu(B) - \delta - \eta \geq \mu(B) - 2\delta.
      \end{split}  
    \]
    From the other side, we have
    \[
        \frac{1}{|F'|}\sum_{g\in F'} \delta_B(gx) \leq \frac{1}{(1-\eta)|F|} \sum_{g\in F} \delta_B(gx) \leq \frac{1}{1-\eta}(\mu(B) + \delta) \leq \mu(B) + 2\delta.
    \]
\end{proof}

The above results enable a very useful construction: given an ergodic measure $\mu$, we can take a disjoint quasitiling $\CT$ of $G$ (whose tiles are relatively large subsets of large F\o lner sets), and construct a symbolic element $x_\mu$ by assigning to each tile of $\CT$ a block which is sufficiently close to $\mu$ (the existence of such blocks follows from the ergodic theorem). Every invariant measure supported by the orbit closure of such an $x_\mu$ will also be close to $\mu$ (and can be guaranteed to be arbitrarily close to $\mu$, provided that the blocks used to construct $x_\mu$ are sufficiently close to $\mu$, and the tiles of $\CT$ are sufficiently large subsets of sufficiently late F\o lner sets).
\begin{cor}\label{cor:measure_approximation}
    If $\mu$ is a measure on $(X,\CB)$ and $\eps>0$, then there exist $\delta$ and $n_0$, such that if all the following are satisfied:
    \begin{enumerate}
        \item $\CT$ is a disjoint, $(1-\delta)$-covering quasitiling of $G$,
        \item Every shape of $\CT$ is a $(1-\delta)$-subset of some F\o lner set $F_n$, where $n$ may vary but is always larger than $n_0$,
        \item $x\in X$ is such that for all $T\in\CT$ we have $d(\mu, x(T))<\delta$, 
    \end{enumerate}
    then for any measure $\mu'$ supported by the orbit closure of $x$ we have $d(\mu, \mu')<\eps$.
\end{cor}

\subsection{Entropy}
We will assume the reader is familiar with most definitions and properties related to measure-theoretic entropy, but we note that we use logarithms with base 2 in our calculations. We will just recall the basic formulae for the entropy  $H(P)$ of a partition and the entropy $h(P)$ of the action with respect to the partition as follows: If $P=\set{P_1,P_2,\ldots,P_m}$ is a partition of $X$ into $m$ measurable sets, then for any $g\in G$ we define $gP =\set{gP_1,gP_2,\ldots, gP_m}$, and we define
\[ H(P) = -\sum_{i=1}^m\mu(P_i)\log\mu(P_i),\quad H_n(P)= H\left(\bigvee_{g\in F_n} gP\right),\text{ and }\, h(P) = \lim_{n\to\infty} \frac{1}{\abs{F_n}}H_n(P).\]

\begin{thm}[The Shannon-McMillan-Breiman Theorem]\label{lem:SMB}
    Let $P$ be a finite partition, and assume that $G$ is a discrete amenable group acting ergodically on a measure space $(X, \CB, \mu)$. Assume $F_n$ is a tempered sequence of F\o lner sets, with $\frac{|F_n|}{\log n} \to \infty$. Then, if $P^{F_n}(x)$ denotes the element of $\bigvee_{g\in F_n}gP$ to which $x$ belongs, we have
    \[
        \frac{-\log(\mu(P^{F_n}(x))}{|F_n|} \to h(P) \text{ as } n\to\infty
    \] 
    a.e. and in $L^1(\mu)$.
\end{thm}

In our case, we only use generating partitions, so the right side is always equal to the entropy $h$ of the system. For our purposes it will be more convenient to use convergence in measure, stated precisely as follows: for each $\delta > 0$ there exists $N$ such that for each $n > N$ we have:
\[
   2^{-|F_n|(h + \delta)} \leq \mu(P^{F_n}(x)) \leq 2^{-|F_n|(h - \delta)}
\]
with probability at least $1-\delta$ (meaning that the set of atoms of $P^{F_n}$ for which the above does not hold has joint measure less than $\delta$).

We can obtain similar properties for sufficiently large subsets of F\o lner sets, at the cost of worsening the estimates slightly.

\begin{lem}\label{lem:SMB for subsets}
    Suppose we have a finite set $F \subset G$ such that 
    \[
   2^{-|F|(h + \delta)} \leq \mu(P^{F}(x)) \leq 2^{-|F|(h - \delta)}
    \]
    with probability at least $1-\delta$. Then, there exists $\eta > 0$ (which depends on $\delta$ and $h$, but not on $F$) such that if $F'$ is a $(1-\eta)$-subset of $F$, then:
    \[
   2^{-|F'|(h + 3\delta)} \leq \mu(P^{F'}(x)) \leq 2^{-|F'|(h - 3\delta)}.
    \]
    with probability at least $1-3\delta$.
\end{lem}

\begin{proof}
    
    Denote $s=\abs{P}$ and let $\eta = \min(\frac{2\delta}{h+3\delta}, \frac{\delta}{\log s})$. Such a choice of $\eta$ guarantees that $(1-\eta)(h+3\delta)\geq(h+\delta)$, and also that $\eta\log s \leq \delta$. Thus, if $\abs{F'}$ is a $(1-\eta)$-subset of $F$, we have the estimate $\abs{F'}(h+3\delta)\geq \abs{F}(h+\delta)$, and since $P^{F'}(x) \supseteq P^{F}(x)$, the following inequality holds with probability at least $1-\delta$:
    \[
        \mu(P^{F'}(x)) \geq \mu(P^{F}(x)) \geq 2^{-|F|(h+\delta)} \geq 2^{-|F'|(h+3\delta)}.
    \] 
    In the other direction, first note that since $F'\subset F$, the atom $P^{F'}(x)$ is a union of disjoint sets of the form $P^{F}(\hat{x})$, where $P^{F'}(x)=P^{F'}(\hat{x})$ and the only difference between those disjoint sets is caused by elements of $F\setminus F'$. It follows that $P^{F'}(x)$ corresponds to at most $s^{|F\setminus F'|}$ different blocks with domain $F$. 

     Let $M$ be the number of blocks $B$ with domain $F$ such that 
     \[
        \mu([B]) \geq 2^{-|F|(h-2\delta)}.
     \]
     Since the joint measure of such blocks is at most $\delta$, we see that  
     \[M\cdot 2^{-|F|(h-2\delta)} \leq \delta, \]
     and so
     \[M\cdot 2^{-|F|(h-\delta)} \leq \delta\cdot 2^{-|F|\delta}, \]
     Now we estimate the probability of the event that $P^{F'}(x)$ has a subset of the form $P^{F}(\hat{x})$ with measure at least $2^{-|F|(h-2\delta)}$. There are at most $M$ blocks which have such a subset, and so we can estimate their total measure by the following expression (where the first term is an upper estimate of the total size of the subsets of measure at least $2^{-|F|(h-\delta)}$, and the second term is the upper estimate on the total measure of all the other subsets):
    \[\begin{split}
        \delta + Ms^{|F\setminus F'|}2^{-|F|(h-\delta)} &\leq \delta + s^{\eta\abs{F}}\cdot M\cdot2^{-|F|(h-\delta)}  \leq \delta(1+s^{\eta\abs{F}}2^{-|F|\delta}) =\\ 
        & =  \delta(1+2^{\eta\log s\abs{F}}2^{-|F|\delta})= \delta(1+2^{\abs{F}(\eta \log s-\delta)}) \leq 2\delta.\end{split}\]
    For all other $x$, and thus with probability at least $1-2\delta$, we have:
    \[
        \mu(P^{F'}(x)) \leq s^{|F\setminus F'|}2^{-|F|(h-2\delta)} \leq 2^{-|F'|(h-3\delta)}.
    \]
    
    Ultimately (combining this with the previous inequality, which holds with probability greater than $1-\delta$), we have:
    \[
        2^{-|F'|(h+3\delta)} \leq  \mu(P^{F'}(x)) \leq 2^{-|F'|(h-3\delta)}
    \]
    with probability at least $1-3\delta$.
\end{proof}

We will relate the notions of entropy theory to the approximate inclusions and equalities defined in the previous section, using the following lemma:
\begin{lem}\label{lem:inclusion_from_entropy}
    For any countable partitions $P$, $Q$ of a probability space $(X,\CB,\xi)$, and any $\epsilon > 0$, there exists $\delta >0 $ such that if $H_\xi(P|Q) < \delta$, then $P \subset_{\epsilon} Q \mod \xi$.
\end{lem}
\begin{proof}
First recall the following version of the rectangle rule: if $\mu$ is a probability measure on a set $\Omega$ and $f \geq 0$ is a measurable function on $\Omega$ such that $\int f \text{ d}\mu < ab$ for some $a > 0$ and $b > 0$, then
       $\mu(\{\omega\in\Omega:f(\omega) \geq a\}) < b$.

   For the convenience of the reader, we recall the conditional entropy formula:
    
    \[
        H_\xi(P|Q) = \sum_{B \in Q} \xi(B) H_{B}(P),
    \]
    where $H_B(P)$ is the entropy of $P$ with respect to the conditional measure $\xi_B(A)=\frac{\xi(B\cap A)}{\xi(B)}$ (if $\xi(B)=0$, we set $H_B(P)=0$). Now let $\delta < \frac{1}{9} \epsilon^2$. If $H_\xi(P|Q)<\delta$, the rectangle rule implies that $H_{B}(P) < \sqrt{\delta}$ on sets $B\in Q$ of joint measure at least $1-\sqrt{\delta}$.  Let $t$ denote the largest value of $\xi_B(A)$ for $A\in P$ (we will refer to such an $A$ as the \emph{dominating set}). Since $H_{B}(P)$ can be assumed to be arbitrarily small (in particular, less than $\log 2$), we can assume that $t > \frac{1}{2}$. The upper estimate on $H_{B}(P)$ yields a lower estimate on $t$ as follows:
    \[
        \sqrt{\delta} \geq H_{B}(P) \geq -t\log t - (1-t)\log(1-t) \geq 2t(1-t) \geq 1-t,
    \]
    so $t\geq 1 - \sqrt{\delta}$ (we have used the inequality $-\log t \geq 1-t$). 
    
    Now, for any $A \in P$ we define
    \[
    B_A =   \{\bigcup B:H_{B}(P) < \sqrt{\delta},\ A\text{ is the dominating set for }B\}
    \] and $B_P = \bigcup_{A\in P} B_A$. Now we are ready to estimate $\xi(A \bigtriangleup B_A)$:
    
    \begin{align*}
        \xi(A \bigtriangleup B_A) &= \xi((A \bigtriangleup B_A)\cap B_A) + \xi((A \bigtriangleup B_A)\cap (B_P\setminus B_A)) + \xi((A \bigtriangleup B_A)\cap (X\setminus B_P )) \\ 
        &\leq \xi(B_A \setminus A) + \xi(A \cap (B_P\setminus B_A)) + \xi(A\setminus B_P) \\ 
        &\leq (1-t)\xi(B_A) + (1-t)\xi(B_P) + \sqrt{\delta} \leq \sqrt{\delta} + \sqrt{\delta} + \sqrt{\delta} = 3\sqrt{\delta} < \epsilon,
    \end{align*}
    and thus $P \subset_{\epsilon} Q$ modulo $\xi$.
\end{proof}

Finally, there is one more technical result we will need (compare \cite{G}, for the $\mathbb{Z}$-action case). 

\begin{lem}
\label{lem:higher entropy}
    Let $(X, \CA, \mu, G)$ and $(Y, \CB, \nu, G)$ be ergodic symbolic dynamical systems, where the alphabet of $X$ has cardinality $s$. Let $Z=X\times Y$ be the product space, and let $\xi$ be any joining of $\mu$ and $\nu$, i.e. any invariant measure on $Z$ such that its marginal measures $\xi^X$, $\xi^Y$ are $\mu$ and $\nu$, respectively. For every $\epsilon\in(0,1)$ there exists an ergodic measure $\xi_1$ on $Z$, such that
    \begin{enumerate}
        \item $\xi_1^Y=\xi^Y$
        \item $d(\xi_1,\xi)<2\epsilon$
        \item $h_{\xi_1^X}(X) \geq h_{\xi^X}(X) + \epsilon(\log(s)-h_{\xi^X}(X))$. In particular, if $h_{\xi^X}(X) < \log s$, then  $h_{\xi_1^X}(X) > h_{\xi^X}(X)$. 
    \end{enumerate}
    
\end{lem}

\begin{proof}
    Let $P$ and $Q$ denote the finite partitions of $Z$ induced by symbols of $X$ and $Y$ respectively\footnote{These are the same partitions as the ones defined more precisely on page \pageref{partitions_def}.}.  The basic idea is to perturb the $X$-values independently of $\xi$. Let $X_0=\{0,1,2,\ldots,s\}^G$ be a shift space with group action and the Bernoulli measure $\mu_0$ such that $\mu_0(0)=1-\epsilon$ and $\mu_0(i)=\frac{\epsilon}{s}$ for $i=1,2,\ldots,s$. Put $W = Z \times X_0$, $\zeta = \xi \times \mu_0$ and define the mapping $\Phi: W \rightarrow W$ in the following way
    \[
        \Phi(x,y,c) = (x^1,y^1,c^1)
    \]
    where $y^1=y$ and $c^1=c$, and
    \[
x_g^1 = \left\{ \begin{array}{ll}
x_g & \textrm{ if } c_g=0\\
c_g & \textrm{ if } c_g \neq 0.

\end{array} \right.
\]

Define $\zeta_1 = \zeta \circ \Phi^{-1}$ and $\xi_1 = {\zeta_1}^{X\times Y}$. It is clear that $\zeta_1$ is $G \times G \times G$-invariant, and so $\xi_1$ is $G \times G$-invariant.

Recall that we use the following metric on the set of joinings:
\[ d(\xi_1,\xi_2)=\sum_{n=1}^\infty\frac{1}{2^n}\frac{1}{\abs{\CB_n}}\sum_{D\in\CB_n}\abs{\xi_1(D)-\xi_2(D)}, \]
where $\CB_n=\bigvee_{g\in F_n} g(P \vee Q)$.
Let $F = \{g_1,g_2,\ldots,g_m\}$ be a finite subset of $G$, and let $D \in \bigvee_{g\in F} g(P \vee Q)$. Now we need to estimate $|\xi(D) - \xi_1(D)|$. It is clear that
\begin{align*}
    \xi_1(D)&=\zeta_1(D\times X_0) = \\ &= \zeta_1(D \times [c_{g_1}=0, c_{g_2}=0, \ldots, c_{g_m}=0]) + \zeta_1(D \times [c_{g_1}=0, c_{g_2}=0, \ldots, c_{g_m}=0]^{\complement}).
\end{align*}
Let $\delta_0$ be the probability measure on $X_0$ concentrated in the single point which has $0$'s on all coordinates. If we now write $\zeta_0=\xi\times \delta_0$, we naturally have 
\[\xi(D)=\zeta_0(D\times X_0)=\zeta_0(D\times [c_{g_1}=0, c_{g_2}=0, \ldots, c_{g_m}=0]).\]
Thus
\begin{align*}
\abs{\xi_1(D)-\xi(D)}&\leq \abs{\zeta_0(D\times [c_{g_1}=0, c_{g_2}=0, \ldots, c_{g_m}=0])-\zeta_1(D\times [c_{g_1}=0, c_{g_2}=0, \ldots, c_{g_m}=0])}+\\
&+\zeta_1(D \times [c_{g_1}=0, c_{g_2}=0, \ldots, c_{g_m}=0]^{\complement})
\end{align*}
By the construction of $\zeta_0$ and $\zeta_1$ respectively, we have
\[\zeta_0(D\times [c_{g_1}=0, c_{g_2}=0, \ldots, c_{g_m}=0])=\xi(D),\] 
and
\[\zeta_1(D\times [c_{g_1}=0, c_{g_2}=0, \ldots, c_{g_m}=0])=\xi(D)(1-\eps)^m,\] 
and so
\begin{align*}
\abs{\xi_1(D)-\xi(D)}& \leq  \xi(D)(1-(1-\eps)^m)+ \zeta_1(D \times [c_{g_1}=0, c_{g_2}=0, \ldots, c_{g_m}=0]^{\complement}).
\end{align*}
Summing  over all $D\in  \bigvee_{g\in F} g(P \vee Q)$, we have
\begin{align*}
\sum_{D}\abs{\xi_1(D)-\xi(D)}& \leq \sum_{D}\left(\xi(D)(1-(1-\eps)^m)\right)+\\
&\phantom{\leq}+\sum_{D}\zeta_1(D \times [c_{g_1}=0, c_{g_2}=0, \ldots, c_{g_m}=0]^{\complement})=\\
&=(1-(1-\eps)^m)+\zeta_1(Z\times [c_{g_1}=0, c_{g_2}=0, \ldots, c_{g_m}=0]^{\complement})=&\\
&=2(1-(1-\eps)^m)\leq 2m\eps= 2\abs{F}\eps,
\end{align*}
where the final estimate follows from Bernoulli's inequality. This gives us the following estimate for the metric:
\begin{align*} d(\xi,\xi_1)&=\sum_{n=1}^\infty\frac{1}{2^n}\frac{1}{\abs{\CB_n}}\sum_{D\in\CB_n}\abs{\xi(D)-\xi_1(D)}\leq\\ &\leq \sum_{n=1}^\infty\frac{1}{2^n}\frac{1}{\abs{\CB_n}}\cdot2\abs{F_n}\eps \leq 2\eps.
\end{align*}



The next part is the entropy. First we analyze  $\xi_1^X(x^1_{g_1} = s_1, x^1_{g_2} = s_2, \ldots, x^1_{g_m} = s_m)$, where $\{g_1,g_2,\ldots, g_m\} = F$ are group elements and $(s_1,s_2,\ldots,s_m) = \overline{s}$ are corresponding symbols.
Let $(a_1,a_2,\ldots,a_k) = \overline{a}$ be the indices such that $c_{g_{a_i}} = 0$ and let $(b_1, b_2, \ldots, b_{m-k}) = \overline{b}$ be the indices where $c_{g_{b_i}} = s_{b_i} \neq 0$. Of course $\{g_{a_1},g_{a_2},\ldots,g_{a_k}, g_{b_1}, g_{b_2}, \ldots, g_{b_{m-k}}\} = F$. We have
\begin{align*}
   &\zeta_1^{X\times X_0}(x^1_{g_1} = s_1, x^1_{g_2} = s_2, \ldots, x^1_{g_m} = s_m, c_{g_{a_1}} = 0, \ldots, c_{g_{a_k}} = 0, c_{g_{b_1}} = s_{b_1}, \ldots, c_{g_{b_{m-k}}} = s_{b_{m-k}}) = \\ 
   &= (1-\epsilon)^k \Big(\frac{\epsilon}{s}\Big)^{m-k}\xi^X(x_{g_{a_1}} = s_{a_1}, \ldots, x_{g_{a_k}} = s_{a_k}) = \\
   &= (1-\epsilon)^k \epsilon^{m-k}\frac{\xi^X(x_{g_{a_1}} = s_{a_1}, \ldots, x_{g_{a_k}} = s_{a_k})}{s^{m-k}}.
\end{align*}
To simplify notation, we denote the set $\{x_{g_{a_1}} = s_{a_1}, \ldots, x_{g_{a_k}} = s_{a_k}\}$ as $D_{\overline{a},\overline{s}}^k$, and then we can write
\begin{align*}
    \xi_1^X(x^1_{g_1} = s_1, \ldots, x^1_{g_m} = s_m) = \sum_{k=0}^m \sum_{(a_i)_{i=1}^{k}} (1-\epsilon)^k \epsilon^{m-k}\frac{\xi^X(D_{\overline{a},\overline{s}}^k)}{s^{m-k}}.
\end{align*}
Using concavity of the function $f(x)=-x\log x$ we get
\begin{align*}
    f(\xi_1^X(x^1_{g_1} = s_1, \ldots, x^1_{g_m} = s_m)) &= f\Big(\sum_{k=0}^m \sum_{(a_i)_{i=1}^{k}} (1-\epsilon)^k \epsilon^{m-k}\frac{\xi^X(D_{\overline{a},\overline{s}}^k)}{s^{m-k}}\Big) \geq \\ 
    &\geq \sum_{k=0}^m \sum_{(a_i)_{i=1}^{k}} (1-\epsilon)^k \epsilon^{m-k} f\Big( \frac{\xi^X(D_{\overline{a},\overline{s}}^k)}{s^{m-k}}\Big).
\end{align*} 
Recall that $P^F = \bigvee_{g\in F} gP$, and so we can write
\begin{align*}
    H_{\xi_1^X}(P^F) &= \sum_{(s_j)_{j=1}^m} f(\xi_1^X(x^1_{g_1} = s_1, \ldots, x^1_{g_m} = s_m)) \geq \\ 
    &\geq \sum_{(s_j)_{j=1}^m} \sum_{k=0}^m \sum_{(a_i)_{i=1}^{k}} (1-\epsilon)^k \epsilon^{m-k} f\Big( \frac{\xi^X(D_{\overline{a},\overline{s}}^k)}{s^{m-k}}\Big).
\end{align*}

Now the value $f\Big( \frac{\xi^X(D_{\overline{a},\overline{s}}^k)}{s^{m-k}}\Big)$ depends only on the indices $a_i$ and the values $s_{a_i}$. For a fixed choice of $a_i$ and $s_{a_i}$, we are left with the specific set of indices $b_j$ of cardinality $m-k$, and each value $s_{b_j}$ could be any symbol from the alphabet. In the end, $f\Big( \frac{\xi^X(D_{\overline{a},\overline{s}}^k)}{s^{m-k}}\Big)$ appears exactly $s^{m-k}$ times. Furthermore, denote $P^{\overline{a}} = \bigvee_{{a_i} \in \overline{a}} g_{a_i}P$ and observe that \[
h_{\xi^X}(X) = h_{\xi^X}(P) = \inf_{F\subset G} \frac{1}{|F|}H_{\xi^X}(P^{F}) \leq \frac{1}{\abs{{\overline{a}}}}H_{\xi^X}(P^{{\overline{a}}}),
\]
whence
\[H_{\xi^X}(P^{\overline{a}})\geq \abs{\overline{a}} h_{\xi^X}(X)=k h_{\xi^X}(X).\] 
Using all of the above, we have
\begin{align*}
     H_{\xi_1^X}(P^F) &\geq \sum_{(s_i)_{i=1}^{n}} \sum_{k=0}^m \sum_{(a_i)_{i=1}^{k}} (1-\epsilon)^k \epsilon^{m-k} f\left( \frac{\xi^X(D_{\overline{a},\overline{s}}^k)}{s^{m-k}}\right) = \\ 
     &= \sum_{k=0}^m \sum_{(a_i)_{i=1}^{k}} \sum_{(s_{a_i})_{i=1}^k} (1-\epsilon)^k \epsilon^{m-k}s^{m-k} f\left( \frac{\xi^X(D_{\overline{a},\overline{s}}^k)}{s^{m-k}}\right) = \\ 
     &= \sum_{k=0}^m \sum_{(a_i)_{i=1}^{k}} \sum_{(s_{a_i})_{i=1}^k} (1-\epsilon)^k \epsilon^{m-k} \xi^X(D_{\overline{a},\overline{s}}^k)  \left(-\log (\xi^X(D_{\overline{a},\overline{s}}^k)) +\log s^{m-k}\right) = \\ 
     &=  \sum_{k=0}^m \sum_{(a_i)_{i=1}^{k}} \sum_{(s_{a_i})_{i=1}^k} (1-\epsilon)^k \epsilon^{m-k} \left[f(\xi^X(D_{\overline{a},\overline{s}}^k)) + \xi^X(D_{\overline{a},\overline{s}}^k)\log s^{m-k}\right] = \\
     &= \sum_{k=0}^m \sum_{(a_i)_{i=1}^{k}} (1-\epsilon)^k \epsilon^{m-k} \left[H_{\xi^X}(P^{{\overline{a}}}) +\log s^{m-k}\right] \geq \\
     &\geq \sum_{k=0}^m \binom{m}{k} (1-\epsilon)^k \epsilon^{m-k} \left[ kh_{\xi^X}(X) +\log s^{m-k}\right].
\end{align*}

Finally we get
\begin{align*}
    \frac{1}{m}H_{\xi_1^X}(P^F) 
    &\geq \sum_{k=0}^m \binom{m}{k} (1-\epsilon)^k \epsilon^{m-k} \Big[\frac{kh_{\xi^X}(X) +(m-k)\log s}{m}\Big]   \\
    & =\sum_{k=0}^m \binom{m}{k} (1-\epsilon)^k \epsilon^{m-k} \Big[\frac{m h_{\xi^X}(X) +(m-k)(\log s-h_{\xi^X}(X))}{m}\Big] \\      & = h_{\xi^X}(X)+(\log s-h_{\xi^X}(X))\cdot\sum_{k=0}^m \binom{m}{k}\cdot\frac{m-k}{m}(1-\epsilon)^k \epsilon^{m-k} \\
    & = h_{\xi^X}(X) + \epsilon\cdot(\log s - h_{\xi^X}(X)),
\end{align*}

and so
\[
    h_{\xi_1^X}(X) = \inf_{F\subset G}\frac{1}{\abs{F}}H_{\xi_1^X}(P^F) \geq h_{\xi^X}(X) + \epsilon(\log s - h_{\xi^X}(X)).
\]

\end{proof}

\section{The main results}
Throughout this section, $(Y, \CB, \nu, G)$ will be a symbolic dynamical system over the countable alphabet $\set{1,2,3,\ldots}$, with $\nu$ being a fixed ergodic measure on $\CB$, with finite entropy $h_\nu$ (note that any ergodic dynamical system with finite entropy has a countable generating partition, and thus is isomorphic to a system of the form $(Y, \CB, \nu, G)$ as defined here). In addition, $(X, \CA, G)$ will be a symbolic dynamical system over the finite alphabet ${1,2,\ldots,s}$, where $s\geq 2$ and $\log s \geq h_\nu$, but the measure on $X$ may (and will) vary. Finally, $(Z,\CC, G)$ is the product space with the product $\sigma$-algebra and the product action. 

As stated in the introduction, we will use joining techniques as a unified approach to proving the following three classical theorems:
\begin{thm}[The finite generator theorem]
    If $(Y,\CB,\nu, G)$ is an ergodic dynamical system with finite entropy $h_\nu$, then there exists a measure $\mu$ on $\CA$ such that the systems $(X,\CA,\mu,G)$ and $(Y, \CB, \nu, G)$ are isomorphic.
\end{thm}
\begin{thm}[The homomorphism theorem]\label{thm:sinai}
    If $(Y,\CB,\nu, G)$ is an ergodic dynamical system with finite entropy $h_\nu$, and $\mu$ is a Bernoulli measure on $(X,\CA)$ such that $h_\nu\geq h_\mu$, then there exists a homomorphism (factor map) from $(Y,\CB,\nu, G)$ onto $(X,\CA,\mu, G)$.
\end{thm}
\begin{thm}[The isomorphism theorem]
    If $(Y,\CB,\nu, G)$ and $(X,\CA,\mu, G)$ both have Bernoulli measures of equal entropy, then they are isomorphic to each other (in this case we assume that the alphabet of $Y$ is finite rather than countable).
\end{thm}

Before we rephrase these theorems in terms of joinings, we need to define several auxiliary objects in the product space $Z=X\times Y$ with the product $\sigma$-algebra $\CC$.

For $i=1,2,\ldots, s$, let $P_i$ be the subset of $Z$ defined as
\[\label{partitions_def}
P_i = \{z=(x,y) \in Z: x_e = i\}.
\]
Similarly, for any $j\geq 1$, let
\[
    Q_j = \{z=(x,y) \in Z: y_e=j\},\]
    and for any $l\geq 1$ let
    \[\ Q_l^{(l)} = \{z=(x,y) \in Z: y_e\geq l\}.
\]
We can now define the following partitions of $X\times Y$: let $P=\set{P_1,P_2,\ldots P_s}$, let $Q^{(l)} = \{Q_1, \ldots, Q_{l-1}, Q_l^{(l)}\}$, and let $Q = \{Q_1,Q_2,\ldots\}$. Moreover, we denote by $\CP$ the $\sigma$-algebra generated by all $gP$ for $g \in G$, and we define $\CQ^{(l)}$  and $\CQ$ analogously. Observe that $\CP = \CA \times \set{\emptyset, Y}$, and $\CQ = \set{\emptyset, X}  \times \CB$, thus $(Z,\CP)$ and $(Z,\CQ)$ are isomorphic to $(X,\CA)$ and $(Y,\CB)$, respectively.

Note that if $\mu$ and $\nu$ are ergodic $G$-invariant measures on $\CA$ and $\CB$ respectively, then any isomorphism $\phi:Y\to X$ between $(Y,\CB,\nu, G)$ and $(X,\CA,\mu, G)$ gives rise to an ergodic joining $\xi$ of $\mu$ and $\nu$, supported by the graph of $\phi$. Conversely, any joining $\xi$ of $\mu$ and $\nu$ such that $\CP = \CQ \mod \xi$ is induced by an isomorphism between $X$ and $Y$. An analogous statement is true for homomorphisms, with the equality modulo $\xi$ replaced by the inclusion $\CP \subset \CQ \mod \xi$. If $\xi$ is a joining induced by an isomorphism (homomorphism), we will sometimes refer to $\xi$ itself as an isomorphism (homomorphism).

We will now define the two Baire spaces of fundamental importance to our construction: denote the space of all probability measures on $(Z,\CC)$ by $\CM$, then let
\[
    \CM_0 = \{\xi \in \CM : \xi \text{ is invariant and ergodic, } \xi^Y = \nu,\, h_{\xi^X}(X) \geq h_{\nu}(Y)\};
\]
and for a fixed ergodic measure $\mu$ on $(X,\CA,G)$, define
\[
    \CM_1 = \{ \xi\in \CM : \xi \text{ is invariant and ergodic, } \xi^Y=\nu,\,\xi^X = \mu\}.
\]
It is easily verified that these sets are nonempty, since any ergodic component of the product measure belongs to both of them. In addition, they are both Baire subsets of $\CM$ (which is itself a compact metrizable space), i.e. they have the Baire property that countable intersections of (relatively) open dense subsets are dense (and hence nonempty). We can now restate our three theorems as follows:

\begin{thm}[The finite generator theorem]\label{thm:krieger_joinings}
    If $h_{\nu}(Y) < \log s$, then
    \[
        \CM_0^{*} = \{\xi \in \CM_0: \xi \text{ is an isomorphism}\}
    \]
    is a countable intersection of dense open subsets of $\CM_0$.
\end{thm}
\begin{thm}[The homomorphism theorem]
    If $h_{\nu}(Y) = h_{\mu}(X)$ and $\mu$ is Bernoulli, then
    \[
        \CM_1^{*} = \{\xi \in \CM_1: \xi \text{ is an homomorphism}\}
    \]
    is a countable intersection of dense open subsets of $\CM_1$.
\end{thm}
Note that it is sufficient to prove the homomorphism theorem for the case of equal entropies: indeed, if  $h_\nu(Y)>h_\mu(X)$ then we can appropriately split the states of $X$ to obtain another Bernoulli system $X'$ of the same entropy as $Y$; the original $X$ is then clearly a homomorphic image of $X'$, all we need is prove the above theorem with $X$ replaced by such an $X'$.

\begin{thm}[The isomorphism theorem]
    If $\mu$ and $\nu$ are Bernoulli and $h_{\nu}(Y) = h_{\mu}(X)$, then
    \[
        \CM_2^{*} = \{\xi \in \CM_1: \xi \text{ is an isomorphism}\}
    \]
    is a countable intersection of dense open subsets of $\CM_1$.
\end{thm}

Observe that Theorem 3.6 is an immediate corollary of Theorem 3.5: applying the latter twice, with the roles of $X$ and $Y$ interchanged, we obtain two sets whose intersection has to be nonempty. Any $\xi$ in such an intersection corresponds to an (almost everywhere) invertible map between $X$ and $Y$, which is a homomorphism in both directions, and thus is an isomorphism. Therefore, we only need to prove the first two theorems.

\subsection*{Proof of the finite generator theorem}\phantom{ }

Our goal is to represent $\CM_0^{*}$ as the intersection of a countable collection of dense open subsets of $\CM_0$. For fixed $k,l\geq 1$, let $V_{k,l}$ be the set of all $\xi \in \CM_0$ such that
\[
    \CQ^{(l)} \subseteq_{\frac{1}{k}} \CP \text{ mod } \xi  \text{   \  and\ \ }  P \subseteq_{\frac{1}{k}} \CQ \text{ mod } \xi.
\]

Note that
\[ 
    P \subset_{\frac{1}{k}} \mathcal{Q} \text{ mod } \xi 
\]

if and only if for some $n$ and $l$ we have
\[
    P \subset_{\frac{1}{k}} \bigvee_{g\in F_n} gQ^{(l)} \text{ mod } \xi,
\]
which is an open condition (involving a finite number of strict inequalities between measures of sets). A similar argument applies to the other inclusion, and so it follows that $V_{k,l}$ is an open set. Furthermore, if $\xi \in \bigcap_{k,l} V_{k,l}$ then 
\[
     P \subseteq \CQ \text{ mod } \xi \text{   \  and\ \ } Q \subseteq \CP \text{ mod } \xi,
\] 
therefore $\xi$ is an isomorphism. Ultimately,
\[
    \bigcap V_{k,l} = \CM_0^{*}.
\]

It remains to show that each $V_{k,l}$ is dense in $\CM_0$. Thus, we fix some $\xi \in\CM_0$, and recall that
\[
    \xi^Y = \nu,\ h_{\xi^X}(X) \geq h_{\nu}(Y),
\]
and also $\xi$ is ergodic under the product group action. For any given $k,l$ we need to prove that there exists a $\widetilde{\xi}\in V_{k,l}$ as close as we want to $\xi$. Since it is easy to see that the sets $V_{k,l}$ are decreasing in $k$ and $l$, we may assume that $k$ and $l$ are arbitrarily large. Now, it suffices to show that for every $n_0\in\mathbb{N}$ and $\epsilon>0$, we can find a $\widetilde{\xi}\in V_{k,l}$ such that for any atom $A$ of $\bigvee_{g\in F_{n_0}}g(P\vee Q)$ we have:
\[
    |\xi(A)-\widetilde{\xi}(A)| < \epsilon.
\]
We will construct such a $\widetilde{\xi}$ directly, over a number of intermediate steps:
\subsubsection*{Step 1}
By lemma \ref{lem:inclusion_from_entropy}, there exists an $\eps_k$ such that if $H_{\overline{\xi}}(P|\CQ)<\eps_k$ and $H_{\overline{\xi}}(Q^{(l)}|\CP)<\eps_k$, then $\overline{\xi}\in V_{k,l}$. Since we now work with a fixed $l$, and all measure properties with respect to $\CQ$ need only be verified with respect to $\CQ^{(l)}$, we will simplify the notation by writing $Q$ and $\CQ$ instead of $Q^{(l)}$ and $\CQ^{(l)}$, respectively (implicitly treating $Y$ as a system over a finite alphabet ${1,2,\ldots,l}$).
\subsubsection*{Step 2}
By lemma \ref{lem:higher entropy}, we can find a joining $\xi_1$ in $\CM_0$ which is arbitrarily close to $\xi$, but such that $h_{\xi_1^X}>h_{\xi^X}\geq h_{\nu}$. Therefore we can without loss of generality assume that $h_{\xi^X}>h_{\xi^Y}=h_\nu$.

\subsubsection*{Step 3} Let $d = h_{\xi^X}(X) - h_{\xi^Y}(Y)$. It follows from step 2 that $d>0$. A combination of lemma \ref{lem:SMB for subsets} (the corollary of the Shannon-McMillan-Breiman theorem), and lemma \ref{lem:dyntil} (the existence of quasitilings whose shapes are arbitrarily large subsets of arbitrarily large F\o lner sets) implies that for every $y\in Y$ we can obtain a disjoint quasitiling $\CT_y$ such that $\set{\CT_y:y\in Y}$ is a factor of $Y$, and the shapes of the quasitilings are large enough to apply the SMB theorem in $X$, $Y$ and $X\times Y$ with arbitrary accuracy at the same time. Specifically, we can require that every shape $S$ of any $\CT_y$ satisfies the following (for an arbitrarily small $\delta$, which we specify later, and $\delta'$ chosen so that corollary \ref{cor:measure_approximation} is satisfied for $\xi$ with $\frac{\eps}{2}$ in place of $\eps$):
\begin{enumerate}
\label{inequalities}
    \item the set of blocks $B$ in $Y$, with domain $S$, such that 
        \[
            \xi^Y(B) \geq 2^{-|S|(h_{\xi^Y} + \delta)},
        \]
         has total joint measure at least $1-\delta$ (denote the set of these blocks as $\mathbf{B}$),
        
    \item the set of blocks $A'$ in $X$, with domain $S$, such that 
        \[
            \xi^X(A') \leq 2^{-|S|(h_{\xi^X} - \delta)},
        \]
        has total joint measure at least $1-\delta$ (denote the set of these blocks as $\mathbf{A}'$),

    \item the set of blocks $A'\times B$ in $X\times Y$, with domain $S$, such that $d(A'\times B,\xi)<\delta'$, and
    \[
            2^{-|S|(h+\delta)} \leq \xi(A'\times B) \leq 2^{-|S|(h-\delta)},
        \]where $h=h(\xi)$, has total joint measure at least $1-\delta$.
                \end{enumerate}
In addition to any other requirements, we can without loss of generality assume that $\delta < \frac{\epsilon}{18|F_{n_0}|}$, and $\delta<\frac{d}{12}$.

The above conditions will be satisfied as long as the parameters $\eta$ and $N$ in lemma \ref{lem:dyntil} are chosen sufficiently small and sufficiently large, respectively. In addition we can require that $\eta < \frac{\epsilon}{12|F_{n_0}|}$, and also that $(2\delta+2\eta)\cdot\max\set{\log s, \log l}<\eps_k$. Once $\eta$ is fixed, we know that every shape of every $\CT_y$ is a $(1-\eta)$-subset of one of finitely many F\o lner sets which we can enumerate as $F_{k_1}, F_{k_2},\ldots, F_{k_{N_{\CT}}}$. Note that $N_{\CT}$ depends only on $\eta$, and all the $k_i$'s can be assumed to be arbitrarily large.
\subsubsection*{Step 4} Similarly to the case of $\mathbb{Z}$-actions, we want to encode $\CT_y$ within specially constructed elements $\overline{x}$ of $X$, so that the content of any such $\overline{x}$ allows us to reconstruct $\CT_y$. In the case of $\mathbb{Z}$ actions it is sufficient to ensure that a specific marker block (in \cite{BKS} it is a long block consisting of $1$'s) occurs in $\overline{x}$ with controlled gaps, leaving a large amount of freedom in specifying the content of $\overline{x}$ between the occurrences of the marker block (the only requirement is that we do not introduce an extra occurrence of a marker block, but that still leaves enough freedom to construct measures with all necessary properties). Our method of encoding follows a similar intuition (we will place marker blocks at centers of tiles of $\CT_y$, leaving much freedom in filling the remainder of the tiles), but since $G$ does not have a natural ordering and is not necessarily Abelian, the construction becomes more challenging. 

Since the disjoint quasitiling $\CT_y$ is a factor of a (not necessarily disjoint) quasitiling by $N_{\CT}$ F\o lner sets (where $N_{\CT}$ depends only on the constant $\eta$ determined in the previous step), we will need exactly $N_{\CT}$ marker blocks: every tile $T$ of every $\CT_y$ has a unique representation $T=Sc$, where $S$ is a $(1-\eta)$-subset of $F_{k_i}$ for a uniquely determined $i$. Thus, to encode $\CT_y$ within some $\overline{x}\in X$ it suffices to ensure that $\overline{x}$ has the $i$'th marker block at every such $c$, and no marker block occurs anywhere else in $\overline{x}$.

The domain $D$ of each marker block $M_i$ will consist of some finite set $D_0$ (chosen such that the marker will have a small probability of occurring ``randomly'', and such that certain overlaps of the domain with itself are not possible), and additional coordinates $g_1,g_2,\ldots,g_{N_{\CT}}$ (which will be used to distinguish between the variants of the marker block). The marker block $M_i$ itself will have the symbols $1$ on all coordinates in $D_0$ and at $g_i$, and it will have the symbol $2$ at all coordinates $g_j$ for $j\neq i$. Then, to ensure that $\overline{x}(Dg)=M_i$ if and only if $\CT_y$ includes a tile $Sg$ associated with the F\o lner set $F_{k_i}$, we will simply change at least one $1$ to a $2$ within every domain of the form $Dg'$ for other $g'$. If the total measure of the blocks $M_i$ is small enough (and it can be arbitrarily small), the number of necessary changes will be small enough to leave sufficiently many blocks available to obtain other desired properties. However, \emph{a priori} we also need to consider the subtle challenge arising from the possibility that $Dg=Dg'$ for some $g'\neq g$ (and then we could not change $\overline{x}(Dg')$ since that would destroy the necessary marker block). Fortunately, we can avoid this by putting some additional care into constructing the set $D$, so that the equality $Dg=Dg'$ is never possible for $g\neq g'$. The precise reasoning and construction is presented in the following lemma:
\begin{lem}
\label{lem:marker}
    For each $N \in \mathbb{N}$ and $\delta_M>0$ there exists a finite set $D$ and blocks $M_1,M_2,\ldots M_N$ with  domain $D$ such that:
    \begin{enumerate}
        \item The total measure of the cylinders corresponding to these blocks is at most $\delta_M$;
        \item If $x(D)=M_i$ and the symbol $1$ occurs nowhere in $x(D^{-1}D \setminus D)$, then $x(D)$ is the only occurrence of any $M_i$'s in $x(D \cup D^{-1}D)$.
    \end{enumerate}
\end{lem}

\begin{proof}
    Let $D = D_0 \cup \{g_1,g_2,\ldots, g_N\}$ where:
    \begin{itemize}
        \item the measure of the block with domain $D_0$ which has the symbol $1$ at all coordinates is at most $\delta_M$,
        \item $|D_0|$ is a prime number,
        \item $e\in D_0$ and $g_0\in D_0$ such that $\text{ord}(g_0) \neq |D_0|$ (the order of the element $g_0$),
        \item for $i=1,2,\ldots,N$, we have $g_i \notin D_0^2$ , and also $D_0g_i \cap D_0 = \emptyset$. 
    \end{itemize}
    
    Define $M_i$ as a block having the symbol $1$ at coordinates from $D_0$ and at $g_i$, and the symbol $2$ at all the other coordinates in $D$. The total measure of such blocks is less than the measure of the block having $1$'s on the coordinates from $D_0$, and so less than $\delta_M$. 
    
    Now suppose that $x(D) = M_i$, and $x$ does not have the symbol $1$ at any other coordinate from $D\cup D^{-1}D$. It is clear that if $g \in D^{-1}D \setminus D$ or $g = g_j$ for $j\neq i$, then $x(Dg)$ cannot be equal to any $M_i$, because the  center symbol of any such block is not $1$. For $g=g_i$ this is also impossible, because $x(Dg_i)$ has only one symbol $1$. Now suppose there is a $g \in D_0 \setminus e$ such that $x(D_0g)$ consists only of $1$'s. We know that $g_i \notin D_0g$ which means that $D_0g = D_0$ and $\text{ord}(g) = |D_0|$. Since $D_0$ is finite, $|D_0|$ is prime and $e\in D_0$, we can write 
    \[
        D_0 = \{g^j: j=0,1,\ldots |D_0|-1 \}
    \]
    and so $g_0 = g^j$ for some $j$. This is a contradiction, because it implies $\text{ord}(g_0) = |D_0|$, and we assumed otherwise.
\end{proof}
\subsubsection*{Step 5}
We will create a new joining $\overline{\xi}$ as the image of $\xi$ under a map which codes every $(x,y)$ to a point $(\overline{x},\overline{y})$ in the following way: $\overline{y}=y$, and $\overline{x}$ is constructed based on $y$, using a coding procedure which is ,,nearly'' invertible (thus $\overline{\xi}$ will be ,,nearly'' an isomorphism joining). To this end, we need to establish injective block codes (dictionaries), between blocks in $Y$ and in $X$, one such dictionary for each shape of the quasitiling $\CT_y$. The existence of such dictionaries will follow from the following variant of the marriage lemma: if $\mathbf{B}$ and $\mathbf{A}$ are finite sets, $\sim$ is a relation between $\mathbf{B}$ and $\mathbf{A}$, and there exists a $K>0$ such that every $B\in \mathbf{B}$ is in relation with at least $K$ elements of $\mathbf{A}$, and also every $A\in \mathbf{A}$ is in relation with at most $K$ elements of $\mathbf{B}$, then there exists an injective mapping $f:\mathbf{B}\to\mathbf{A}$.

Let $S$ be a shape of the quasitiling $\CT_y$ (and thus a $(1-\eta)$-subset of $F_{k_i}$ for some $i\in \set{1,2,\ldots, N_{\CT}}$). Let $\mathbf{B}$ and $\mathbf{A}'$ be defined as in step 3 above. For each $A'\in \mathbf{A}'$ we define $A = \Psi(A')$ by changing some symbols in $A'$, in order to make sure that $A$ has exactly one marker subblock:
\begin{itemize}
    \item Based on lemma \ref{lem:marker} there exists a domain $D$ and a collection of marker blocks $M_1,\ldots,M_{N_{\CT}}$ with domain $D$, so that the total measure of these blocks is arbitrarily small. Let $\delta_M$ denote the upper estimate of this total measure of marker blocks. We will put specific constraints on this value later, but here we note that it can be made arbitrarily small with respect to $\eta$ and $\delta$ (since these two values can be specified before $\delta_M$).
    \item Replace symbols at coordinates from the set $D$ from lemma \ref{lem:marker} (the domain of the marker blocks) with the marker block $M_i$.
    \item For any other occurrence of $M_1,M_2,\ldots, M_{N_{\CT}}$ within $A'$, replace one instance of the symbol $1$ with $2$ (thus eliminating all other occurrences of marker blocks within $A'$).
    \item Replace all $1$'s with $2$'s at coordinates from the set $\{g\in S: Dg \nsubseteq S\}$ (this set can be assumed to have arbitrarily small cardinality relative to $S$).
\end{itemize}
The mapping $\Psi$ is clearly not injective. However, for any $A$ the cardinality of the set $\{ A' \in \mathbf{A}': \Psi(A') = A\}$ is at most
\[
    s^{|D|} + \sum_{i = 1}^j \binom{|S|}{i}2^i,
\]
where $j$ is the maximal number of occurrences of the blocks $M_1,M_2,\ldots,M_{N_{\CT}}$ in the various $A'$'s. 
Since all the measures in question are ergodic, and $\abs{S}$ can be assumed to be arbitrarily large notwithstanding any other conditions, we can assume that the number $|\frac{j}{|S|}-\delta_M|$ is arbitrarily small and so $\frac{j}{|S|}$ is arbitrarily small compared to $\delta$. (Note that to define the marker blocks, we only need to specify $N_{\CT}$ and $\delta_M$, which in turn depend only on $\eta$ and $\delta$, and thus we can actually choose them before we fix the minimal possible size of $S$. Thus we can assume that the size of $S$ is so large that all the blocks with domain $S$, in which the marker blocks occur too often, have negligible total measure and can be excluded from $A'$ already in step 3). In addition we can assume $|D| < \frac{\delta}{\log s}|S|$. Ultimately, we can choose $\delta_M$ and then $\abs{S}$ so that
\[
    s^{|D|} = 2^{|D|\log s} \leq 2^{\delta|S|},
\]
and also
\[ \frac{2j}{|S|}+ \frac{j}{|S|}\log\left(\frac{3|S|}{j}\right) \leq \delta.\]
For such $A'$ we have the following estimate (also using the inequality $\binom{n}{k}\leq \left(\frac{ne}{k}\right)^k\leq\left(\frac{3n}{k}\right)^k$):
\[
    \sum_{i = 1}^j \binom{|S|}{i}2^i \leq j\binom{|S|}{j}2^j \leq j\Big(\frac{3|S|}{j}\Big)^j 2^j = 2^{j+\log j + j\log(\frac{3|S|}{j})} \leq 2^{|S|(\frac{2j}{|S|}+ \frac{j}{|S|}\log(\frac{3|S|}{j})}) \leq 2^{\delta|S|}.
\]
Together (for large enough $S$) we get the inequality
\[
    s^{|D|} + \sum_{i = 1}^j \binom{|S|}{i}2^i \leq 2^{2\delta|S|}.
\]
Let $\mathbf{A} = \{\Psi(A'): A' \in \mathbf{A}'\}$ (note that every block from $\mathbf{A}$ has exactly one marker subblock). We are ready to define a relation $\sim$  as follows: $A \in \mathbf{A}$ and $B\in \mathbf{B}$ are \emph{related} if there exists an $A' \in \mathbf{A}'$ with $\Psi(A')=A$ such that the pair $(A',B)$ considered as an atom of the partition 
\[
   \bigvee_{g\in S}g(P\vee Q),
\]
satisfies both conditions specified in the third item of step 3. Recall the relevant estimates from that set of inequalities:
\begin{enumerate}
    \item \[
            \xi^Y(B) \geq 2^{-|S|(h_{\xi^Y} + \delta)}
        \]
        
    \item \[
            \xi^X(A') \leq 2^{-|S|(h_{\xi^X} - \delta)}
        \]
    \item \[
            2^{-|S|(h+\delta)} \leq \xi(A'\times B) \leq 2^{-|S|(h-\delta)}.
        \]
\end{enumerate}

Since $B$ is the union of disjoint sets of the form $A'\times B$, we can combine the inequalities $(3)$ and $(1)$ (and the fact that the mapping $A'\mapsto A$ is at most $2^{2\delta|S|}$-to-one), to conclude that every $B\in \mathbf{B}$ is in relation with at least (recall that in step 3 we chose $\delta<\frac{d}{12}$)
\[
    2^{|S|(h-h_{\xi^{Y}}-2\delta-2\delta)} \geq  2^{|S|(h-h_{\xi^{Y}}-\frac{d}{3})} =: K
\]
elements of $\mathbf{A}$. Similarly, inequalities $(2)$ and $(3)$ together imply that every $A\in \mathbf{A}$ is in relation with at most
\[
    2^{|S|(h-h_{\xi^{X}}+2\delta)} \leq 2^{|S|(h-h_{\xi^{X}}+\frac{d}{3})} < K
\]
elements of $\mathbf{B}$ (recall that $d$ was defined in step 3 as $h_{\xi^Y}-h_{\xi^X}$). Therefore by the marriage lemma we can match each element of $\mathbf{B}$ (i.e. a block from $Y$) to an element of $\mathbf{A}$ (i.e. a modified version of a block from $X$) in an injective fashion. We denote the family of such functions as $\Phi_S$, where $S$ is the domain of our blocks. In addition, for the blocks $B\notin \mathbf{B}$, set $\Phi_S(B)=A$, where $A$ is chosen arbitrarily from the range of $\Phi_S$.

We are now ready to define the mapping that will transport $\xi$ to a joining $\overline{\xi}$, which will hold all required properties, possibly except for the inequality $h_{\overline{\xi}^X} \geq h_{\overline{\xi}^Y}$. To do so, we modify a point $(x,y)$ to a point $(\overline{x},\overline{y})$ as follows:
\begin{enumerate}
    \item $\overline{y}=y$,
    \item for each tile $T$ of $\CT_y$ let $S$ be the shape of $T$ and let $B$ be the block occurring in $y$ at coordinates $T$, then set $\overline{x}(T)=\Phi_S(B)$,
    \item for $g\in G$ which do not belong to any tile of $\CT_y$, we set $\overline{x}(g)=2$.
\end{enumerate}

We need to establish the following:
\begin{enumerate}
    \item $|\xi(A)-\overline{\xi}(A)| < \frac{\epsilon}{2}$ for any atom $A$ of $\bigvee_{g\in F_{n_0}}g(P\vee Q)$,
    \item $H_{\overline{\xi}}(P|\CQ) < \epsilon_k$,
    \item $H_{\overline{\xi}}(Q|\CP) < \epsilon_k$.
\end{enumerate}

The first inequality follows directly from corollary \ref{cor:measure_approximation}, since every $(\overline{x},\overline{y})$ was constructed by placing blocks from $\mathbf{A'}\times\mathbf{B}$ within all tiles of $\CT_y$, and every block from $\mathbf{A'}\times\mathbf{B}$ approximates $\xi$ within $\eta$. Thus any invariant measure supported by the orbit closure of any $(\overline{x},\overline{y})$ is within $\frac{\eps}{2}$ of $\xi$.

To estimate $H_{\overline{\xi}}(P|\CQ)$ we need to answer the following question: When is it possible, for some $(\overline{x}, \overline{y})$, to determine $\overline{x}(e)$, knowing the symbols in $\overline{y}(F_n)$ for large $n$? There are only two cases when this may \emph{not} be possible. The first such situation is when $e$ does not belong to any tile of $\CT_{\overline{y}}$ (which is the same as $\CT_{y}$); this happens with probability less than $2\eta$. The second possibility is that $e$ does belong to some tile with shape $S$, but the corresponding block in $\overline{y}$ is not in the domain of the associated mapping $\Phi_S$. The probability of such a situation is less than $2\delta$. Therefore, the only atoms of $\bigvee_{g\in F_n}gQ$ with nonzero contribution to conditional entropy have total measure less than $2\delta+2\eta$. In addition, we chose $\delta$ and $\eta$ so small that $(2\delta+2\eta)\log s <\eps_k$, whence 
\[
    H_{\overline{\xi}}(P|\CQ) \leq H_{\overline{\xi}}(P|\bigvee_{g\in F_n}gQ) \leq
    (2\delta + 2\eta)\log s < \epsilon_k.
\]

The reasoning for  $H_{\overline{\xi}}(Q|\CP)$ is very similar. Recall that in the pair $(\overline{x}, \overline{y})$ the point $\overline{x}$ was constructed so that for every tile $T$ of $\CT_y$ the block $\overline{x}(T)$ has a marker subblock, and these marker subblocks allow us to recreate $\CT_y$ (at least locally) based on $\overline{x}(F_n)$ for large enough $n$. Thus, for large enough $n$, we can determine which (if any) tile of $\CT_{y}$ includes the neutral coordinate $e$, and we can also determine the shape $S$ of such a tile. Since $\Phi_S$ is injective, the only cases when we cannot determine $\overline{y}(e)$ are when $e$ does not belong to any tile (which has probability less than $2\eta$) or $\overline{x}(S)$ is not in the range of $\Phi_S$ (which has probability less than $2\delta$). Thus we get
\[
    H_{\overline{\xi}}(Q|\CP) \leq H_{\overline{\xi}}(Q|\bigvee_{g\in F_n}gP) \leq
    (2\delta + 2\eta)\log l < \epsilon_k.
\]
As stated earlier, while $\overline{\xi}$ obviously has $\nu$ as the marginal on $Y$, the entropy of $\overline{\xi}^X$ may be less than $h_{\nu}$ (which would disqualify $\overline{\xi}$ from $V_{k,l}$). However, we can estimate this potential deficit of entropy. Since
\[ h_{\nu}=h_{\overline{\xi}}(\CQ)= h_{\overline{\xi}}(\CQ|\CP)+h_{\overline{\xi}}(\CP)\leq H_{\overline{\xi}}(Q|\CP)+h(\overline{\xi}^X), \]
we have 
\[ h({\overline{\xi}^X})\geq h_{\nu}-H_{\overline{\xi}}(Q|\CP)\geq h_{\nu}-(2\delta + 2\eta)\log l.\]
Thus, if $(2\delta+2\eta)\log l < \frac{\eps}{2}\cdot(\log s-h_{\nu})$ (which we can safely assume since $s$, $\nu$ and $l$ are known from the start, and $\eps$ is chosen before $\delta$ and $\eta$), lemma \ref{lem:higher entropy} allows us to find a measure $\widetilde{\xi}\in \CM_0^*$ such that $d(\widetilde{\xi},\overline{\xi})<\frac{\eps}{2}$ and $h(\widetilde{\xi}^X)\geq h(\overline{\xi}^X)+\frac{\eps}{2}(\log s - h_{\overline{\xi}^X}(X))\geq h_{\nu}$. Since $V_{k,l}$ is open, for sufficiently small $\eps$ we have a guarantee that such a $\widetilde{\xi}$ will be in $V_{k,l}$.

\subsection*{Proof of the homomorphism theorem}
The arguments in this short section are virtually unchanged from those in \cite{BKS} (since all the necessary tools for amenable groups already exist), however, we include the proof for completeness. For a detailed discussion of finite determinedness, proximity in finite distributions, and connections with the existence of relatively independent joinings, we refer the reader to \cite{OW} and/or \cite{G}; here we will just briefly recall the main notions in our specific context.

Let $(Y,\CB,\nu,G)$ be a system with positive, finite entropy $h_\nu$, and let $(X,\CA,\mu,G)$ be a Bernoulli system also of entropy $h_\nu$. Let $P_X=\{P_1,\ldots,P_s\}$ denote the generating partition of $X$ into cylinders, and let $P=\{P_1\times Y,P_2\times Y,\ldots, P_s\times Y\}$ be the corresponding partition in the product space. Recall that 
\[
    \CM_1 = \{ \xi\in \CM : \xi \text{ is invariant and ergodic, } \xi^Y=\nu,\,\xi^X = \mu\},
\]
and
    \[
        \CM_1^{*} = \{\xi \in \CM_1: \xi \text{ is a homomorphism}\},
    \]
    and our goal is to show that $\CM_1^*$ is dense in $\CM_1$.
Let 
\[
\CM_{1,n} = \set{\xi\in\CM_1 : P\subset_{\frac{1}{n}}\CB \,\mod \xi}.
\]
It is easy to see that these sets are open, that $\CM_1=\bigcap_{n}\CM_{1,n}$ (since $P_X$ is a generating partition), and it remains to prove that they are dense in $\CM_1$. Thus, for some fixed $\xi \in \CM_1$, $\eps>0$ and $n\geq 1$, we need to find a $\widetilde{\xi}\in\CM_{1,n}$ such that $d(\widetilde{\xi},\xi)<\eps$. Again, we will construct such a $\widetilde{\xi}$ in stages:
\subsubsection*{Step 1}
Before we proceed, recall the following facts about the $\overline{d}$ distance between any two ergodic measures $\mu$ and $\mu_1$ on $X$:
\begin{itemize}
    \item Following \cite{G}, we will define the metric $\overline{d}$ as
    \[ \overline{d}(\mu,\mu_1)=\inf\set{\sum_{A\in P}\lambda\left((A\times X) \sd (X\times A)\right)}, \]
    where $\lambda$ ranges over the joinings between $\mu$ and $\mu_1$. This infimum is in fact a minimum and is always attained by an ergodic joining. The topology induced by $\overline{d}$ is in general stronger than the weak-$*$ topology.
    \item If $\mu$ is a Bernoulli measure, then it is finitely determined. For our purposes this means that for every $\eps>0$ there exists a $\delta>0$ such that if $\mu_1$ is another measure on $X$, such that $|h(\mu_1)-h(\mu)|<\delta$, and $\mu_1$ is $\delta$-close to $\mu$ in the weak-$*$ metric, then $\overline{d}(\mu,\mu_1)<\eps$. 
\end{itemize}
\subsubsection*{Step 2} Since $\CM_1$ is a subset of $\CM_0$, theorem \ref{thm:krieger_joinings} implies that we can find a joining $\xi_1\in\CM_0^*$ arbitrarily close to $\xi$ (in particular, closer than $\frac{\eps}{2}$, although this is not the only proximity condition that we require). Any such $\xi_1$ will have the following properties:
\begin{itemize}
    \item $\xi_1^Y=\nu$,
    \item $\xi_1$ is an isomorphism between $(Y,\CB,\nu)$ and $(X,\CA,\xi_1^X)$,
    \item $h(\xi_1^X)=h(\xi_1^Y)=h_\nu$.
\end{itemize}
\subsubsection*{Step 3} Sufficient proximity of $\xi_1$ to $\xi$ in the weak-$*$ metric on $\CM$ implies proximity of $\xi_1^X$ to $\xi^X$ (which is equal to $\mu$) in the weak-$*$ metric on $X$. Since $(X,\CA,\mu,G)$ is Bernoulli, and thus finitely determined, such proximity (combined with the equality of entropies) means that $\overline{d}(\xi^X_1,\mu)$ can also be made arbitrarily small. Denote $\xi_1^X$ as $\mu_1$. 
\subsubsection*{Step 4} Sufficient proximity between $\mu_1$ and $\mu$ in $\overline{d}$ implies that there exists an ergodic joining $\rho$ on $X\times X$, projecting onto $\mu$ on the first coordinate, and $\mu_1$ on the second, such that the quantity
\[ \rho\left( \bigcup_{A\in \bigvee_{g\in F_N} gP_X} A\times A \right) \]
can be assumed to be greater than $1-\eps_1$ for arbitrarily small $\eps_1$ and arbitrarily large $N$. In particular this means that for every $A\in \bigvee_{g\in F_N} gP_X$ we have $\mu_1(A)-\rho(A\times A)<\eps_1$.
\subsubsection*{Step 5} The systems $(X\times X,\rho)$ and $(X\times Y, \xi_1)$ have $(X,\mu_1)$ as a common factor, and thus they have a relatively independent joining over $(X,\mu_1)$, which can be represented as a measure $\zeta$ on $X\times X\times Y$ (with marginals $\mu$, $\mu_1$ and $\nu$, respectively). Let $\widetilde{\xi}$ be the projection of $\zeta$ onto the outer coordinates, i.e. a measure on $X \times Y$ such that $\widetilde{\xi}(A\times B)=\zeta(A\times X\times B)$. Let $\rho_x$ and $(\xi_1)_x$ denote the disintegrations over $(X,\mu_1)$ of $\rho$ and $\xi_1$, respectively. For any $A\in\bigvee_{g\in F_N} gP_X$, $B\in\CB$, note the following:
\[ \xi_1(A\times B) = \int_A(\xi_1)_x(B)d\mu_1(x) =\int_X \chi_A(x)\cdot(\xi_1)_x(B)d\mu_1(x), \]
\[ \widetilde{\xi}(A\times B) = \zeta(A\times X\times B)= \int_X \rho_x(A)\cdot(\xi_1)_x(B)d\mu_1(x).\]
Therefore 
\[\begin{split}
 \abs{\xi_1(A\times B)-\widetilde{\xi}(A\times B) } &= \abs{\int_X \chi_A(x)\cdot(\xi_1)_x(B)d\mu_1(x)-\int_X \rho_x(\widetilde{A})\cdot(\xi_1)_x(B)d\mu_1(x)} \\
 &\leq \int_X\abs{\chi_A(x)\cdot(\xi_1)_x(B)-\rho_x(A)\cdot(\xi_1)_x(B)}d\mu_1(x) \\
 &\leq \int_A\abs{\chi_A(x)\cdot(\xi_1)_x(B)-\rho_x(A)\cdot(\xi_1)_x(B)}d\mu_1(x) \\
 &= \int_A (1-\rho_x(A))\cdot(\xi_1)_x(B)d\mu_1(x)\leq \int_A (1-\rho_x(A))d\mu_1(x) \\
 &=\mu_1(A)-\rho(A\times A)<\eps_1.
\end{split}\]
Since $\eps_1$ can be assumed to be arbitrarily small, and $N$ arbitrarily large (note that all that is required for all our conditions to hold is that the $\xi_1$ we choose at the start is sufficiently close to $\xi$), this means that we can have $d(\widetilde{\xi},\xi_1)<\frac{\eps}{2}$, and thus $d(\widetilde{\xi},\xi)<\eps$.

In addition, since $P\subset \CB\, \mod\xi_1$ (in fact, they are equal$\mod\xi_1$), choosing $\eps_1$ to be sufficiently small will guarantee that $P\subset_{\frac{1}{n}} \CB\,\mod\widetilde{\xi}$. Also $\widetilde{\xi}$ has $\mu$ as the marginal on $X$ (immediately from definition), and is ergodic (even though the relatively independent joining $\zeta$ need not itself be ergodic, almost all of its ergodic components factor onto $\widetilde{\xi}$, which implies ergodicity of the latter), and thus $\widetilde{\xi}\in\CM_{1,n}$.\qed

%
%
%

\end{document}